\documentclass[12pt]{amsart}
\usepackage{amssymb}
\usepackage[all]{xy}
\usepackage[raiselinks=false,colorlinks=true,citecolor=blue,urlcolor=blue,linkcolor=blue,bookmarksopen=true,pdftex]{hyperref}
\newtheorem{theorem}{Theorem}[section]
\newtheorem{lemma}[theorem]{Lemma}
\theoremstyle{definition}
\newtheorem{definition}[theorem]{Definition}

\newtheorem{proposition}[theorem]{Proposition}

\newtheorem{remark}[theorem]{Remarks}
\def\aut{\operatorname{Aut}}
\def\out{\operatorname{Out}}
\def\homeo{\operatorname{Homeo}}
\def\supp{\textrm{supp}}
\def\fix{\textrm{Fix}}
\begin{document}

\title[Twisted conjugacy in R. Thompson's group $T$] {Twisted conjugacy in  Richard Thompson's group  $T$}
\author[ D. L. Gon\c{c}alves] {Daciberg L. Gon\c{c}alves}
\address{Deptametno de Matem\'atica - IME, Universidade de S\~ao Paulo\\ Caixa Postal 66.281 - CEP 05314-970, S\~ao Paulo - SP, Brasil}
\email{dlgoncal@ime.usp.br}
\author[P. Sankaran]{Parameswaran Sankaran}
\address{The Institute of Mathematical Sciences\\
CIT Campus, Taramani,\\ 
Chennai 600113, India}
\email{sankaran@imsc.res.in}


\maketitle
\noindent
{\bf Abstract} {\it Let $\phi:\Gamma\to \Gamma$ be an automorphism of a group $\Gamma$.  We say that 
$x,y\in \Gamma$ are in the same $\phi$-twisted conjugacy class and write $x\sim_\phi y$ if there exists an element $\gamma\in \Gamma$ such that $y=\gamma x\phi(\gamma^{-1})$.  This is an equivalence relation on $\Gamma$ and is called the $\phi$-twisted conjugacy.  Let $R(\phi)$ denote the number of  $\phi$-twisted conjugacy classes in $\Gamma$.  If $R(\phi)$ is infinite for all $\phi\in \aut(\Gamma)$, we say that $\Gamma$ has the $R_\infty$-property.    The  purpose of this note is to show that the Richard Thompson group $T$ has the $R_\infty$ property.}

\section{Introduction} 

Let $\Gamma$ be a group and let $\phi:\Gamma\to \Gamma$ be an endomorphism.  Then $\phi$ determines an action $\Phi$ of 
$\Gamma$ on itself where, for  $\gamma\in \Gamma$ and $x\in \Gamma$, we have $\Phi_\gamma(x)=\gamma x \phi(\gamma^{-1})$.  The orbits of this action are called the $\phi$-twisted conjugacy classes. Note that when $\phi$ is the identity automorphism, the orbits are the usual conjugacy classes of $\Gamma$.   We denote by $\mathcal{R}(\phi)$ the set of all $\phi$-twisted conjugacy classes and by $R(\phi)$ the cardinality $\#\mathcal{R}(\phi)$ of $\mathcal{R}(\phi)$.   
We say that $\Gamma$ has the $R_\infty$-property if $R(\phi)=\infty,$ that is if $\mathcal{R}(\phi)$ is infinite, for every automorphism $\phi$ of $\Gamma$.   

The problem of determining which groups have 
the $R_\infty$-property---more briefly the $R_\infty$-problem---has attracted the attention of many researchers 
after it was discovered that all non-elementary Gromov-hyperbolic groups have the $R_\infty$-property.  See \cite{ll} and \cite{felshtyn}.  It is particularly interesting when the 
group in question is finitely generated or countable. 
The notion of twisted conjugacy arises naturally in fixed point theory, representation theory, algebraic geometry and number theory.    In the recent years the $R_\infty$-problem has emerged as an active research area. 
The  problem is particularly interesting because there does not seem 
to be a uniform approach to its resolution. 
A variety of techniques and ad hoc arguments from several branches of mathematics have been used in the solve this problem depending on the group under consideration.  These include (but not restricted to) combinatorial group theory, 
geometric group theory, homological algebra, $C^*$-algebras, and algebraic groups.

Recall that Richard Thompson introduced three groups 
$F, T, $ and $V$ in 1965 in an unpublished hand-written 
manuscript.    The group $T$ is the first example of a finitely 
presented infinite simple group.  
In this note we give an elementary proof that $T$ is an $R_\infty$-group. 
\begin{theorem}\label{main}
The Richard Thompson group $T$ has the $R_\infty$-property.
\end{theorem}
We shall describe the groups $F$ and $T$ more fully in \S2, leaving out $V$.  We prove the above theorem in 
\S3.   
 The groups $F$ and $T$ have been 
 generalized by K. S. Brown \cite{BrownFinite} to obtain families of finitely presented groups $F_{n,\infty}, F_n, T_{n,r}, n\ge 2, r\ge 1$ where $F=F_{2}=F_{2,\infty}$ and $T=T_{2,1}$.     Denoting 
any one of them by $F_n$ the group $F_{n,\infty}$ is 
isomorphic to a certain subgroup of $F_n$ of index $n-1$. 
Furthermore $F_{n,r}$ is a subgroup of $T_{n,r}$.   

The 
group $V$ has been generalized by G. Higman 
\cite{HigmanFPSG} to obtain an infinite family of finitely presented groups. But 
we shall not consider them in this paper.

The $R_\infty$-property for $F$ was shown by 
Bleak, Fel'shtyn and Gon\c{c}alves \cite{BlFeGon}.
It has been established by Gon\c{c}alves and Kochloukova \cite[Corollary 4.2]{gk} that the groups $F_{n,\infty}$ have the $R_\infty$-property.  In the last section we make a few comments about the $R_\infty$-property for the groups $F_n$ and $T_{n,r}$.

After this paper was submitted, Collin Bleak had brought our attention to the paper \cite{bmv} of Burillo, Matucci and Ventura where it is shown, among other things, that $T$ has the $R_\infty$-property.  (They also obtain a new proof of the $R_\infty$-property for the group $F$.)  Their proof and the proof given here are based on the same idea of constructing elements with specified number of components of fixed point sets.   We hope  
that Lemmas \ref{finiteorder} and \ref{torsion} which were used in our proof may be useful in other contexts as well.  
 
\noindent
{\bf Acknowledgments:}   We thank J. Burillo for pointing out a misquote in an earlier version of this paper in the statement of Theorem \ref{brin}(i),  based on which we had erroneously claimed that our proof of Theorem \ref{main} also establishes the $R_\infty$-property for $F$.  
We thank Collin Bleak for bringing to our notice the paper \cite{bmv}.    
The first author is indebted to Bleak, A. Fel'shtyn, and J. Taback for fruitful discussions about the Thompson groups.   The first  author has been partially supported by Fapesp project  Tem\'atico Topologia Alg\'ebrica, Geometrica e Diferencial no 2012/24454-8. 
This project was initiated during the visit of the second author to the University of S\~ao Paulo in August 2012. He thanks the first author for the invitation and the warm hospitality.  He is also thankful to the 
organizers of the XVIII Brazilian Topology Meet (EBT-2012) for the invitation and financial support, making 
the visit to Brazil possible.

\section{Richard Thompson's groups $F$ and $T$}

In this section we give a description of Thompson 
groups $F$ and $T$.   The group $F$ consists 
of all piecewise linear (PL) homeomorphisms of $\mathbb[0,1]$ 
with at most a finite set of break points (i.e., points of 
non-differentiability) which are contained in the dyadic rationals $\mathbb{Z}[1/2]$ and having slopes (at points of differentiability) 
in the multiplicative group $\langle 2\rangle=\{2^n\mid n\in \mathbb{Z}\}$.   Note that elements of $F$ are orientation 
preserving.  It is known that the group $F$ is generated by two elements 
$A$ and $B$ defined as follows: 
\[
A(x)=\left\{ \begin{array}{lr}
        x/2, & 0\le  x\le 1/2,\\
        x-1/4, & 1/2\le x\le 3/4,\\
        2x-1, & 3/4\le x\le 1.\\
        \end{array}\right. 
\]
and 
\[B(x)=\left\{ \begin{array}{lr}
x, & 0\le x\le 1/2,\\
x/2+1/4, & 1/2\le x\le 3/4,\\
x-1/8, & 3/4\le x\le 7/8,\\
2x-1, & 7/8\le x\le 1.\\ 
\end{array}\right.
\]
Indeed one has a presentation 
\[F=\langle A, B\mid [AB^{-1}, A^{-1}BA], [AB^{-1}, A^{-2}BA^{-2} ]\rangle. \]

The group $T$ consists of all PL-homeomorphisms of 
the circles $\mathbb{S}^1=I/\{0,1\}$ 
which have at most a finite set of break-points contained in $\mathbb{Z}[1/2]$ and having slopes contained in $\langle 2\rangle$.  Again the elements of $T$ preserve the orientation. 

Any homeomorphism of $[0,1]$ induces a homeomorphism of $\mathbb{S}^1$ and this allows us 
to view $F$ as a subgroup of $T$.  
One has an element $C$ in $T$ which is defined as 
\[ C(x)=\left\{\begin{array}{lr}
x/2+3/4, & 0\le x\le 1/2,\\
2x-1, & 1/2\le x\le 3/4,\\
x-1/4, &3/4\le x\le 1.\\
\end{array}
\right.
\]
It is understood that in the above definition $x$ is read modulo $1$.  

It is known that $T$ is generated by the elements $A, B, C$ with six relations.  Although we will have no 
need for it here, we list below the relations in the said 
presentation for the sake of completeness:  (Note that the first two are the same as 
the defining relations of $F$.)
See \cite{CFP} for details.\\
(1)  $[AB^{-1},A^{-1}BA]=1$,\\
(2)  $[AB^{-1}, A^{-2}BA^2]=1$,\\
(3)  $C=BA^{-1}CB$,\\
(4) $A^{-1}CB.A^{-1}BA=B.A^{-2}CB^2$\\
(5) $CA= (A^{-1} CB)^2$\\
(6) $C^3=1$.

Our proof of Theorem \ref{main} will crucially make use of the following 
result of Brin \cite{BrinCh}.   It is easily seen that 
the reflection map $r$ defined as $r(x)=1-x, x\in [0,1]$, 
induces an automorphism $\rho: T\to T$ defined as $\rho(f)=r\circ f\circ r^{-1}=r\circ f\circ r$. 
We denote by the same symbol $\rho$ the restriction $\rho|_F\in \aut(F)$. 

\begin{theorem} (\cite{BrinCh}) \label{brin}
(i) The group $\out(F)$ of outer automorphisms of $F$ contains an index two subgroup $\out^+(F)$  
isomorphic to $T\times T$. The non-trivial element in the quotient group $\out(F)/\out^+(F)$ is represented 
$\rho\in \aut(F).$\\
(ii) The group of inner automorphisms of $T$ is of index two in $\aut(T)$ and the 
quotient group $\out(T)$ is generated by 
$\rho$. 
\end{theorem}
\section{Proof of Theorem \ref{main}}
 Let $\Gamma$ be a group and let $\phi\in \aut(\Gamma)$. For $\gamma\in \Gamma$, denote by $\iota_\gamma$ the inner automorphism $x\mapsto \gamma x\gamma^{-1}, ~x\in \Gamma$.  We first observe that $R(\phi)=R(\phi\circ \iota_\gamma)$. In fact $y=zx\phi(\iota_\gamma(z^{-1}))$ if and only if $y.\phi(\gamma)=z(x\phi(\gamma))\phi(z^{-1})$. Thus elements of $\mathcal{R}(\phi\circ\iota_\gamma)$ are translation on the left by $\phi(\gamma)$ 
 of  the elements of $\mathcal{R}(\phi)$.   In view of this, 
 to show that $\Gamma$ is an $R_\infty$-group, it suffices to show that $R(\phi)=\infty$ for a set of 
 coset representatives of $\out(\Gamma)$.  In the case 
$\Gamma=T$, in view of  Theorem \ref{brin}  due to Brin,  we need only show that $R(\rho)=\infty$ and $R(id)=\infty$.    The latter equality is established in 
Proposition \ref{conjugacy} as an easy consequence of Lemma \ref{support} below. 

\begin{definition}  {\it Let $X$ be a Hausdorff topological space.\\
(i) The {\em support of}  $f\in \homeo(X)$ is the open set $\supp(f):=\{x\in X\mid f(x)\ne x\}$. \\
(ii) Let $\sigma:\homeo(X) \to \mathbb{N}\cup \{\infty\}$ be  
defined as follows: $\sigma(id)=0$, if $f\ne id$, $\sigma(f)$ is the number of connected components of $\supp(f)$ if it is finite, otherwise $\sigma(f)=\infty$.} 
\end{definition}

\begin{lemma} \label{support} 
Let $\Gamma\subset \homeo(X)$ and let 
$\sigma$ be as defined above.  Suppose that $\theta\in \homeo(X)$ normalizes $\Gamma$.  Then 
$\sigma(f)=\sigma (\theta f \theta^{-1})$.
\end{lemma}
\begin{proof} It is clear that the number of 
connected components of an open set $U\subset X$ remains unchanged under a homeomorphism of $X$.
The lemma follows immediately from the observation 
that $\supp(\theta f \theta^{-1})=\theta(\supp(f))$.  
\end{proof}

The following proposition is well-known.  It follows, for example, from \cite[Remark 3.1]{BlFeGon} for the group $F$ and from \cite[\S5]{CFP} for the group $T$, where certain elements $C_n\in T$ of order $n+2$ for every $n\ge 1$ are explicitly given.   However, we give an elementary unified proof for the sake of completeness.
\begin{proposition} \label{conjugacy}{\em 
The groups $F$ and $T$ have infinitely many 
conjugacy classes.}
\end{proposition}
\begin{proof}
This follows from Lemma \ref{support} on observing that 
$F$ has elements $f$ with $\sigma(f)$ 
any arbitrary prescribed positive integer.  
Since $F\subset T$, the same is true of $T$ as well.
\end{proof}

We need the following lemma.  If $\theta$ is an endomorphism 
of a group $\Gamma$ we denote by $\fix(\theta)$ the fixed subgroup $\{x\in \Gamma\mid 
\theta(x)=x\}$ of $\Gamma$.

\begin{lemma} \label{finiteorder}
Let $\Gamma$ be a group and 
let $\theta\in \aut(\Gamma)$.  Suppose that $\theta^n=\iota_\gamma$.   
Suppose that $\{x^n\gamma\mid x\in Fix(\theta)\}$ is not contained 
in the union of finitely many conjugacy classes of $\Gamma$.  Then $R(\theta)=\infty$.
\end{lemma}
\begin{proof}
Let $x\sim_\theta y$ in $\Gamma$ where $x,y\in \fix(\theta)$. Thus there exists an $z\in \Gamma$ such that 
$y=z^{-1}x\theta(z)$.  Applying $\theta^i$ both sides,  
we obtain $y=\theta^i(z^{-1})x\theta^{i+1}(z)$ as $x,y\in \fix(\theta)$.  
Multiplying these equations successively for $0\le i<n$ and using $\theta^n=\iota_\gamma$, we obtain  
\[ y^n=\prod_{0\le i<n} \theta^i(z^{-1})x\theta^{i+1}(z)
=z^{-1}x^n\theta^n(z)=z^{-1}x^n\gamma z \gamma^{-1}.\]
That is, $y^n\sim_{\iota_\gamma} x^n$.  Equivalently $y^n\gamma$ and $x^n\gamma$ are in 
the same conjugacy classes of $\Gamma$. 
Our hypothesis says that there are infinitely many elements 
$x_k\in \fix(\theta), k\ge 1$, such that the $x_k^n\gamma$  are in 
pairwise distinct $\iota_\gamma$-conjugacy classes of $\Gamma$.   Hence 
we conclude that $R(\theta)=\infty$.  
\end{proof}

We remark that when $\theta\in Aut(\Gamma)$ of order $n$,  we may take $\gamma$ to be the identity.  Therefore
$R(\theta)=\infty$ when $\{x^n \in \Gamma\mid \theta(x)=x\}\subset \Gamma$ is not contained in the union of 
finitely many conjugacy classes of $\Gamma$.  This 
observation leads to the following.

\begin{lemma}\label{torsion}
Let $\theta\in Aut(\Gamma)$ be of order $n<\infty$.  Suppose 
that the set $\mathcal{T}\subset \mathbb{N}$ of orders of torsion elements of $\fix(\theta)$ is unbounded. 
Then $R(\theta)=\infty$.  
\end{lemma}
\begin{proof}  Our hypothesis on $\mathcal{T}$ implies that the 
$\{o(x^n)\mid x\in \fix(\theta)\}\subset \mathbb{N}$ is 
unbounded.  Therefore elements of $\{x^n\mid x\in \fix(\theta)\}$ represent infinitely many distinct 
conjugacy classes of $\Gamma$.  By Lemma \ref{finiteorder}
we conclude that $R(\theta)=\infty$.
\end{proof}

\begin{lemma} \label{iterates}
Suppose that $h:\mathbb{R}\to \mathbb{R}$ is an 
orientation preserving homeomorphism.   Then $\supp(h)=\supp(h^k)$ for any non-zero integer $k$. 
\end{lemma}
\begin{proof} Since $\supp(h)=\supp(h^{-1})$ we may 
assume that $k>0$.
Since $h$ is orientation preserving, it is order preserving. 
Suppose that $x\in \supp(h)$ so that $h(x)\ne x$. Say,  $x<h(x)$.  Then applying $h$ to the inequality we obtain $h(x)<h^2(x)$ so that $x<h(x)<h^2(x)$. Repeating this 
argument yields $x<h(x)<\cdots <h^k(x)$ and so $x\in \supp(h^k)$.   The case when $x>h(x)$ is analogous. Thus $\supp(h)\subset \supp(h^k)$.  On the other hand, if $x\notin \supp(h)$, then $h(x)=x$ and so $h^k(x)=x$ for all $k$.  Therefore equality should hold, completing the proof. \end{proof}

We are now ready to prove our main theorem. 

\noindent
{\it Proof of Theorem \ref{main}:}  
By Theorem \ref{brin}(ii), $\out(T)\cong \mathbb{Z}/2\mathbb{Z}$ generated by $\rho$.  By Proposition 
\ref{conjugacy}, $R(id)=\infty$.   
It only remains to verify that $R(\rho)=\infty$.  We apply 
Lemma \ref{finiteorder} with $\theta=\rho, n=2, \gamma=1$.  
It remains to show that $\fix(\rho)$ has infinitely many elements $h$ such 
that $h^2$ are pairwise non-conjugate.  

Let $k\ge 1$.
Let $f_k\in F\subset T$ be such that $\supp(f_k)\subset (0,1/2)$ and has exactly $k$ components. Thus, $\sigma(f_k)=k$. (It is easy to construct such an 
element.) 
Then $\supp(\rho(f_k))
=\supp(rf_kr^{-1})=r(\supp(f_k))\subset (1/2,1)$ is disjoint from $\supp(f_k)\subset (0,1/2)$.  In particular $f_k.\rho(f_k)=\rho(f_k).f_k=:h_k$, $\supp(h_k)=\supp(f_k)\cup r(\supp(f_k))$ and so $\sigma(h_k)=2k$. 
Moreover, since $\rho^2=1$, we see that $h_k\in Fix(\rho)$.
By 
Lemma \ref{iterates}, we have $\sigma(h_k^2)=\sigma(h_k)=2k$.
It follows that $h_k^2$ are pairwise non-conjugate in $T$, completing the proof. \hfill $\Box$.
\section{Generalized Thompson groups}

 The group $F$ has been 
generalized to yield two families of groups $F_{n,\infty}, F_n$, $n\ge 2$, and the group $T$ likewise has been generalized to a family of groups $T_{n, r},n\ge 2, r\ge 1,$ 
where $F\cong F_2=F_{2,\infty}$ and $T=T_{2,1}$. 
One has inclusions $F_{n,\infty}\subset F_n\subset T_{n,r}$ for all $n\ge 2, r\ge 1$.   Let $\Gamma$ be any one 
of the groups $F_{n,\infty}, F_n, T_{n,r}$.    Then $\Gamma$ is realized as a group of homeomorphisms of $\mathbb{R}$ or $\mathbb{S}^1=\mathbb{R}/r\mathbb{Z}$ according as $\Gamma=F_{n,\infty}, F_n$ or $\Gamma=T_{n,r}$ respectively.   More precisely, one has the following description given in 
\cite[Proposition 2.2.6]{BrinGuzman}.  The group $F_n$ 
is the group of all orientation preserving PL-homeomorphisms of $\mathbb{R}$ having only finitely many break points (that is, points of non-differentiability) such that (i) the break-points are all in $\mathbb{Z}[1/n],$ (ii) the slopes at smooth points are all in the set $\{n^k | k \in \mathbb{Z}\}=:\langle n\rangle$,  (iii) they map the set $\mathbb{Z}[1/n]$ into itself, and, (iv) they are translations by integers near $-\infty$ and $\infty$.   
(For the last condition, see Definition \ref{slopetransl} below.)
The group $F_{n,\infty}$ is the subgroup of $F_n$ which consists 
of those homeomorphisms $f\in F_n$ which maps $\Delta_n$ into itself where $\Delta_n$ is the 
kernel of the unique surjective ring homomorphism $\mathbb{Z}[1/n]\to \mathbb{Z}/(n-1)\mathbb{Z}$.   The group $T_{n,r}\subset \homeo(\mathbb{R}/r\mathbb{Z})$ consists of those elements which are orientation 
preserving and lift to PL-homeomorphisms of $\mathbb{R}$ satisfying conditions (i) to (iii) above.

It turns out that $F\cong F_2=F_{2,\infty}$ and $T\cong T_{2,1}$.   The groups $F_{n,\infty}, F_n$ and $T_{n,r}$ are referred to as the 
generalized Thompson groups.  See \cite{BrinGuzman} 
for a detailed study of these groups and their automorphism groups.  

Brin and Guzm\'an also introduced 
a family of groups $F_{n,j},n\ge 2, j\in \mathbb{Z}$ each of which is isomorphic to $F_{n,\infty}$.  See 
\cite[Lemma 2.1.6 and Corollary 2.3.1.1]{BrinGuzman}.  
Gon\c{c}alves and Kochloukova \cite{gk} have shown,  
using the theory of $\Sigma$-invariants in homological algebra, that the groups $F_{n,0},~n \ge 2,$ 
and hence $F_{n,\infty}$, have the $R_\infty$-property.  In this section 
we show that if $\theta\in \aut(F_n)$ represents a torsion element in the outer 
automorphism group, then $R(\theta)=\infty$.

It was observed by Brin and Guzm\'{a}n, using a deep result of McCleary and Rubin \cite{MR}, that every 
automorphism of  a generalized Thompson group 
is given by conjugation by a homeomorphism of $\mathbb{R}$ or the circle $\mathbb{S}^1$. 
(Such a homeomorphism is not, in general, a PL-homeomorphism!) Invoking this result, 
our proof of Proposition \ref{conjugacy} applies  
equally well to the generalized Thompson groups showing 
that they have infinitely many conjugacy classes.  

\begin{definition} \label{slopetransl}
{\em 
Let $\gamma$ be a PL-homeomorphism of $\mathbb{R}$ with finitely many break points. 
Suppose that $\gamma(t)=at+b$ for $t>0$ large.  We call $a\in \mathbb{R}$ the {\em slope at} $\infty$ and 
$b\in \mathbb{R}$ 
{\em the translation at} $\infty$  and denote them by $\lambda(\gamma)$ and $\tau(\gamma)$ respectively. } 
\end{definition}

We note that $\lambda$ is constant on conjugacy classes of the group of all PL-homeomorphisms of $\mathbb{R}$ with 
finitely many break points.  If $z, h,h'$ are such homeomorphisms and if $\lambda(h)=1=\lambda(h'),$ then $\tau(hh')=\tau(h)+\tau(h')$ and $\tau(zhz^{-1})=\lambda(z).\tau(h)$ as may be verified easily. 

By the description of $F_n$ given above,  $\tau(\gamma)\in \mathbb{Z}$ and $\lambda(\gamma)=1 $ if $\gamma\in F_n$.  Moreover, $\tau(\gamma)\in (n-1)\mathbb{Z}$ if $\gamma\in F_{n,\infty}$.

\begin{proposition} {\em
Let $\theta\in \aut(\Gamma)$ represent an outer automorphism $[\theta]$ of $\Gamma$ of finite order where $\Gamma$ is one of the groups $F_{n,\infty}, F_n, n\ge 2$.  Then $R(\theta)=\infty$.}
\end{proposition}
\begin{proof} Let $o([\theta])=m$. There exists an  $f\in\homeo(\mathbb{R})$ such that $\theta(h)=f h f^{-1}$ for all $h\in \Gamma$. (Cf. \cite[Theorem 1.2.4]{BrinGuzman}, \cite{MR}.)  Since $[\theta]^m=1$ we see that there exists a $\gamma \in \Gamma$ such that $\theta^m(h)=f^m hf^{-m}=\gamma h\gamma^{-1} $ for all $h\in \Gamma$.  

Therefore, setting $g:=\gamma^{-1}f^m\in \homeo(\mathbb{R})$ and  $ghg^{-1}=h $ for all $h\in \Gamma$.   We claim that $g=1$.  To see this, assume that $g\ne 1$ and  choose an interval $U\subset \mathbb{R}$ such that $g(U)\cap U=\emptyset$.  Now let $h\in \Gamma$ be any non-trivial element supported in $U$.  Then $ghg^{-1}$ is supported in $g(U)$. 
This shows that $ghg^{-1}\ne h$, a contradiction.   Hence we conclude that $g=1$ and so $\gamma=f^m$.  In particular $\gamma^k\in \fix(\theta)$ for all $k\in \mathbb{Z}$.

We may assume that $\gamma\ne 1$.  (Otherwise $f\in \homeo(\mathbb{R})$ has order $m$. Hence $m=2$, $\theta(p)=p$ for some $p\in \mathbb{R}$  
and $\theta$ interchanges the intervals $(-\infty, p)$ and $(p,\infty)$. We proceed as 
in the proof of Theorem \ref{main} to see that $R(\theta)=\infty$.)

Since $f^m=\gamma\in \Gamma$ and $\gamma\ne 1$, we have that $\supp(f)=\supp(\gamma)$ equals $\mathbb{R}$ or is a  union of {\it finitely} many open intervals.  In particular $\sigma(\gamma)<\infty$.  

If $\supp(f)$ is not dense, we merely choose elements 
$\gamma_k\in \Gamma$ such that $\supp(\gamma_k)$ has exactly $k$ components  
and $\supp(f)\cap \supp(\gamma_k)=\emptyset $ for all $k\ge 1$.  Then $\theta(\gamma_k)=\gamma_k$ and $\sigma(\gamma_k.\gamma)=k+\sigma(\gamma)$ for all $k$. 
It follows that $\{\gamma_k.\gamma\}_{k\ge 1}$ are in pairwise distinct 
conjugacy classes.

So assume that $\supp(f)=\supp(\gamma)$ is dense.  As remarked above, any element of $\Gamma$ has slope at $\infty$ equals $1$ and translation at infinity an integer, say, $b$.  So we have $\gamma(t)=t+b$ for $t>0$ 
large.  Since $\supp(\gamma)$ is dense, we have $b\ne 0$.

Suppose that $\gamma^r=z\gamma^sz^{-1}$ for some $z\in \Gamma$ where $r,s$ are non-zero integers. Applying $\tau$ we obtain 
$rb=\lambda(z).sb=sb$ as $\lambda(z)=1$.  Hence $r=s$.    
This shows that the elements of $\{\gamma^{mk+1}\mid k\in \mathbb{N}\}$ are in pairwise 
distinct conjugacy classes of $\Gamma$. By Lemma \ref{finiteorder} 
we conclude that $R(\theta)=\infty$.
\end{proof}

\begin{remark}
(i) In the case of the generalized Thompson groups $T_{n,r}$, suppose that $\theta\in \aut(T_{n,r})$ 
represents a torsion element, say of order $m$, in $\out(T_{n,r})$ and that $\theta(x)=fxf^{-1}$ with $f\in \homeo(\mathbb{R}/r\mathbb{Z})$.  Suppose $f^m=\gamma\in T_{n,r}$.  If $\gamma=1$ 
our method of proof of Theorem \ref{main} can be applied to show that $R(\theta)=\infty$. 
However, when $\gamma\ne 1$, it is not clear to us how to find elements of $\fix(\theta)$ 
satisfying the hypotheses of Lemma \ref{finiteorder}.  

(ii)  Our approach yields no information about automorphisms which represent 
non-torsion elements in the outer automorphism group. 

\end{remark}


\begin{thebibliography}{99}
\bibitem{BlFeGon}  C.~Bleak, A.~Fel'shtyn, and D. L. ~Gon\c calves, 
\emph{Twisted conjugacy classes in R. Thompson's group F}, Pacific Journal of Mathematics
\textbf{238} no.1 (2008), 1--6.



\bibitem{BrinCh} Matthew G. Brin,  \emph{The chameleon groups of {R}ichard {J}. {T}hompson: automorphisms
 and dynamics},  Publication Math. Inst. Hautes \'Etudes Sci. {\bf 84} (1997),  5--33.


\bibitem{BrinGuzman}
Matthew~G. Brin and Fernando Guzm{\'a}n, \emph{Automorphisms of generalized
  {T}hompson groups}, Journal of Algebra \textbf{203} (1998), no.~1, 285--348.

\bibitem{BrownFinite}
Kenneth~S. Brown, \emph{Finiteness properties of groups}, Journal of Pure and Applied Algebra \textbf{44} (1987), 45--75.

\bibitem{bmv} J. Burillo, F. Matucci, and E. Ventura, {\it The conjugacy problem for extensions of 
the Thompson's group $F$}, arXiv:1307.6750v2.

\bibitem{CFP} J. W. Cannon,  W. J. Floyd, and W. R. Parry, \emph{ Introductory notes
 on Richard Thompson's groups}, Enseignement Mathematique (2), {\bf 42} 3-4 (1996), 215--256.

\bibitem{felshtyn} A. Felshtyn, {\it The Reidemeister number of any automorphism of a Gromov hyperbolic group is infinite.} Zap. Nauchn. Sem. S.-Peterburg. Otdel. Mat. Inst. Steklov. (POMI) {\bf 279} (2001), Geom. i Topol. {\bf 6}, 229--240.  

\bibitem{gk} D. L. Gon\c{c}alves and D. Kochloukova, 
{\em Sigma theory and twisted conjugacy classes,} Pacific 
Journal of  Mathematics {\bf 247} (2010), 335--352. 

\bibitem{HigmanFPSG}
Graham Higman, {\em Finitely presented infinite simple groups}, Department of
  Pure Mathematics, Department of Mathematics, I.A.S. Australian National
  University, Canberra, 1974, Notes on Pure Mathematics, No. 8 (1974).

\bibitem{ll} G. Levitt and M. Lustig,  {\it Most automorphisms of a hyperbolic group have very simply 
dynamics,} Annales Scient. \'Ecole Normale Superieur {\bf 33} (2000), 
507--517.

\bibitem{MR} S. McCleary and M. Rubin, {\em Locally 
moving groups and the reconstruction problem for 
chains and circles,}  arXiv:math/0510122.
\end{thebibliography}
\end{document}